\documentclass[10pt,psamsfonts]{amsart}
\usepackage{amsmath}
\usepackage{amsthm}
\usepackage{amssymb}
\usepackage{amscd}
\usepackage{amsfonts}
\usepackage{amsbsy}
\usepackage{graphicx}
\usepackage[dvips]{psfrag}
\usepackage{array}
\usepackage{color}
\usepackage{epsfig}
\usepackage{url}

\newcolumntype{L}{>{\displaystyle}l}
\newcolumntype{C}{>{\displaystyle}c}
\newcolumntype{R}{>{\displaystyle}r}

\newcommand{\R}{\ensuremath{\mathbb{R}}}
\newcommand{\N}{\ensuremath{\mathbb{N}}}

\newcommand{\CC}{\mathcal{C}}

\newcommand{\CO}{\ensuremath{\mathcal{O}}}
\newcommand{\ov}{\overline}

\newcommand{\T}{\theta}

\newcommand{\x}{\mathbf{x}}

\newcommand{\sgn}{\mathrm{sign}}
\newcommand{\de}{\delta}
\newcommand{\?}{\!\!\!\!\!\!\!\!\!\!\!\!\!\!}

\def\p{\partial}
\def\e{\varepsilon}

\newtheorem {theorem} {Theorem}
\newtheorem {proposition} [theorem]{Proposition}

\newtheorem {lemma}  [theorem]{Lemma}

\newtheorem {remark} [theorem]{Remark}

\newtheorem {mtheorem} {Theorem}

\begin{document}
\renewcommand{\arraystretch}{1.5}

\author[J. Llibre, D.D. Novaes and M.A. Teixeira]
{Jaume Llibre$^1$, Douglas D. Novaes$^2$  and Marco A.
Teixeira$^2$}

\address{$^1$ Departament de Matematiques,
Universitat Aut\`{o}noma de Barcelona, 08193 Bellaterra, Barcelona,
Catalonia, Spain} \email{jllibre@mat.uab.cat}

\address{$^2$ Departamento de Matematica, Universidade
Estadual de Campinas, Caixa Postal 6065, 13083--859, Campinas, SP,
Brazil} \email{ddnovaes@gmail.com, teixeira@ime.unicamp.br}

\let\thefootnote\relax\footnotetext{Corresponding author Douglas D. Novaes: Departamento de Matematica, Universidade
Estadual de Campinas, Caixa Postal 6065, 13083--859, Campinas, SP,
Brazil. Tel. +55 19 991939832, Fax. +55 19 35216094, email: ddnovaes@gmail.com}


\title[Averaging methods in discontinuous differential systems]
{On the birth of limit cycles for\\ non--smooth dynamical systems}



\maketitle

\begin{abstract}
The main objective of this work is to develop, via Brower degree theory and regularization theory, a variation of the classical averaging method for detecting limit cycles of certain piecewise continuous dynamical systems. In fact, overall results are presented to ensure the existence of limit cycles of such systems. These results may represent new insights in averaging, in particular its relation with non smooth dynamical systems theory. An application is presented in careful detail.
\end{abstract}

\section{Introduction and statement of the main results}

The discontinuous differential systems, i.e. differential equations
with discontinuous right--hand sides, is a subject that has been
developed very fast these last years. It has become certainly one of
the common frontiers between Mathematics, Physics and Engineering. Thus certain phenomena in control systems \cite{Bar}, impact and friction mechanics \cite{Br}, nonlinear oscillations \cite{AVK,M}, economics \cite{H,I}, and biology \cite{Ba,Kr}, are the main sources of motivation of their study, see for more details Teixeira \cite{T}. A recent review appears in \cite{physDspecial}.

\smallskip

The knowledge of the existence or not of periodic solutions is very
important for understanding the dynamics of  differential systems. One of good tools for study the periodic solutions is the
averaging theory, see for instance the books of Sanders and Verhulst
\cite{SV} and Verhulst \cite{V}. We point out that the method of
averaging is a classical and matured tool that provides a useful
means to study the behaviour of nonlinear smooth  dynamical systems.
The method of averaging has a long history that starts with the
classical works of Lagrange and Laplace who provided an intuitive
justification of the process. The first formalization of this
procedure was given by Fatou in 1928 \cite{Fa}. Very important
practical and theoretical contributions in the averaging theory were
made by Krylov and Bogoliubov \cite{BK} in the 1930s and Bogoliubov
\cite{Bo} in 1945. The principle of averaging has been extended in
many directions for both finite- and infinite-dimensional
differentiable systems. The classical results for studying the
periodic orbits of differential systems need at least that those
systems be of class $\CC^2$. Recently Buica and Llibre \cite{BL}
extended the averaging theory for studying periodic orbits to
continuous differential systems using mainly the Brouwer degree theory.

\smallskip

The main objective of this paper is to extend the averaging theory
for studying periodic orbits to discontinuous differential systems
using again the Brouwer degree.

\smallskip

Let $D$ be an open subset of $\R^n$. We shall denote the points of
$\R\times D$ as $(t,x)$, and we shall call the variable $t$ as the
time. Let $h:\R\times D \rightarrow \R$ be a $\CC^1$ function having
the $0\in\R$ as a regular value, and let $\Sigma= h^{-1}(0)$. Given $p\in\Sigma$ we denote its connected component in $\Sigma$ by $\Sigma_p$.

\smallskip

Let $X,Y:\R\times D \rightarrow \R^{n}$ be two continuous vector
fields. Assume that the functions $h$, $X$ and $Y$ are $T$--periodic
in the variable $t$. Now we define a {\it discontinuous piecewise
differential} system
\begin{equation}\label{PKs1}
x'(t)=Z(t,x)=
\begin{cases}
X(t,x)\quad \mbox{if}\quad h(t,x)>0,\\
Y(t,x)\quad \mbox{if}\quad h(t,x)<0.
\end{cases}
\end{equation}
We concisely denote $Z= (X,Y)_h$.

\smallskip

Here we deal with a different formulation for the discontinuous
differential system \eqref{PKs1}. Let $\sgn(u)$ be the sign function
defined in $R\setminus \{0\}$ as
\[
\sgn(u)=
\begin{cases}
\,\,\, 1 \quad \,\, \textrm{if}\quad u>0,\\
-1\quad\textrm{if}\quad u<0.
\end{cases}
\]
Then the discontinuous differential system \eqref{PKs1} can be
written using the function $\sgn(u)$ as
\begin{equation}\label{PKs2}
x'(t)=Z(t,x)=F_1(t,x)+\sgn(h(t,x))F_2(t,x),
\end{equation}
where
\[
\begin{array}{ccc}
F_1(t,x)=\dfrac{1}{2}\left(X(t,x)+Y(t,x)\right)& \textrm{and} &
F_2(t,x)=\dfrac{1}{2}\left(X(t,x)-Y(t,x)\right).
\end{array}
\]

To work with the discontinuous differential system \eqref{PKs2} we
should introduce the regularization process, where the discontinuous
vector field $Z(t,x)$ is approximated by an one--parameter family of
continuous vector fields $Z_{\de}(t,x)$ such that $\lim_{\de\to
0}Z_{\de}= Z(t,x)$.

\smallskip

In \cite{ST} Sotomayor and Teixeira introduced a regularization for
the discontinuous vector fields in $\R^2$ having a line of
discontinuity and, using this technique, they proved generically
that its regularization provides the same extension of the orbits
through the line of discontinuity that the one given by the
Filippov's rules, see \cite{F}. Later on Llibre and Teixeira
\cite{LT2} studied the regularization of generic discontinuous
vector fields in $\R^3$ having a surface of discontinuity, and
proved that $\lim_{\de\to 0}Z_{\de}$ essentially agrees with
Filippov's convention in dimension three.  Finally, in \cite{T}
Teixeira generalized the regularization procedure to finite
dimensional discontinuous vector fields.

\smallskip

In \cite{LST} Llibre, da Silva and Teixeira studied singular
perturbations problems in dimension three which are approximations
of discontinuous vector fields proving that the regularization
process developed in \cite{LT2} produces a singular problem for
which the discontinuous set is a center manifold, moreover, they
proved that the definition of sliding vector field coincides with
the reduced problem of the corresponding singular problem for a
class of vector fields.

\smallskip

In general, a transition function is used in these regularizations to average the vector fields X and Y on the set of discontinuity in order to get a family of continuous vector fields that approximates the
discontinuous one.

\smallskip

A continuous function $\phi:\R\rightarrow\R$ is a {\it transition
function} if $\phi(u)=-1$ for $u\leq-1$, $\phi(u)=1$ for $x\geq 1$
and $\phi'(u)>0$ if $u\in(-1,1)$. The $\phi$--regularization of
$Z=(X,Y)_h$ is the one--parameter family of continuous functions
$Z_{\de}$ with $\de\in (0,1]$ given by
\[
Z_{\de}(t,x)=\dfrac{1}{2}\left(X(t,x)+Y(t,x)\right)+\dfrac{1}{2}
\phi_{\de}(h(t,x))\left(X(t,x)-Y(t,x)\right),
\]
with
\begin{equation}\label{PKf1}
\phi_{\de}(u)=\phi\left(\dfrac{u}{\de}\right).
\end{equation}
Note that for all $(t,x)\in (\R \times D)\backslash \Sigma$ we have
that $\lim_{\de\to 0}Z_{\de}(t,x)=Z(t,x)$.

\smallskip

The formulation \eqref{PKs2} of the discontinuous differential
system \eqref{PKs1} admits a natural regularization. Define the
transition function $\phi$ as
\begin{equation}\label{e1}
\phi(u)=
\begin{cases}
\begin{array}{cl}
1 & \textrm{if}\quad u\geq 1,\\
u & \textrm{if}\quad -1<u<1,\\
-1 &\textrm{if}\quad u\leq-1.
\end{array}
\end{cases}
\end{equation}
Let $\phi_{\de}:\R\rightarrow\R$ be the continuous function defined
in \eqref{PKf1}. It is clear that
\begin{equation}\label{e2}
\lim_{\de \to 0} \phi_{\de}(u)= {\sgn}(u),
\end{equation}
and
\[
Z_{\de}(t,z)=F_1(t,x)+\phi_{\de}(h(t,x))F_2(t,x),
\]
is the {\it $\phi$--regularization} of the discontinuous
differential system \eqref{PKs1}.

\smallskip

As usual $\nabla h$ denotes the gradient of the function $h$, $\p_x h$ denotes the gradient of the function $h$ restricted to
the variable $x$ and $\p_t h$ denotes the partial derivative of the function $h$ with respect to the variable $t$.

\smallskip

Our main results are given in the next theorems. Its proof uses the theory of Brouwer degree for finite dimensional spaces (see the appendix A for a definition of the Brouwer degree $d_B(f,V,0)$), and is based on the averaging theory for non--smooth differential system stated by Buica and Llibre
\cite{BL} (see Appendix B).

\begin{mtheorem}\label{MRt1}
We consider the following discontinuous differential system
\begin{equation}\label{MRs1}
x'(t)=\e F(t,x)+\e^2R(t,x,\e),
\end{equation}
with
\[
\begin{aligned}
&F(t,x)=F_1(t,x)+\sgn(h(t,x))F_2(t,x),\\
&R(t,x,\e)=R_1(t,x,\e)+\sgn(h(t,x))R_2(t,x,\e),
\end{aligned}
\]
where $F_1,F_2:\R\times D\rightarrow\R^n$, $R_1,R_2:\R\times
D\times(-\e_0,\e_0)\rightarrow\R^n$ and $h:\R\times D\rightarrow \R$
are continuous functions, $T$--periodic in the variable $t$ and $D$
is an open subset of $\R^n$. We also suppose that $h$ is a $\CC^1$
function having $0$ as a regular value.

Define the averaged function $f:D\rightarrow\R^n$ as
\begin{equation}\label{MRf1}
f(x)=\int_0^T F(t,x) dt.
\end{equation}
We assume the following conditions.
\begin{itemize}
\item[$(i)$] $F_1,\,F_2,\,R_1,\,R_2$ and $h$ are locally Lipschitz
with respect to $x$;

\item[$(ii)$] there exists an open bounded set $C\subset D$ such that, for $|\e|>0$ sufficiently small, every orbit starting in $C$ reaches the set of discontinuity only at its crossing regions ({\it crossing hypothesis}).

\item[$(iii)$] for $a\in C$ with $f(a)=0$, there exist a neighbourhood
$U\subset C$ of $a$ such that $f(z)\neq 0$ for all $z\in\overline{U}
\backslash\{a\}$ and $d_B(f,U,0)\neq 0$.
\end{itemize}
Then, for $|\e|>0$ sufficiently small, there exists a $T$--periodic
solution $x(t,\e)$ of system \eqref{MRs1} such that $x(0,\e)\to
a$ as $\e\to 0$.
\end{mtheorem}

Theorem \ref{MRt1} is proved in section \ref{s2}.

\smallskip

In order to stablish a theorem with weaker hypotheses we denote by $D_0$ the set of points $z\in D$ such that the map $h_z:t\in[0,T]\mapsto h(t,z)$ has only isolated zeros. Clearly $\textrm{Int}(D_0)\neq \emptyset$.

\begin{mtheorem}\label{MRt2}
In addition to the assumptions of Theorem  \eqref{MRt1} unless condition $(ii)$ we assume the following hypothesis.
\begin{itemize}
\item[$(ii')$] there exists an open bounded set $C\subset D_0$ such that, for $\e>0$ sufficiently small, every orbit starting in $C$ reaches the set of discontinuity only at its crossing regions.
\end{itemize}
Then, for $\e>0$ sufficiently small, there exists a $T$--periodic
solution $x(t,\e)$ of system \eqref{MRs1} such that $x(0,\e)\to
a$ as $\e\to 0$.
\end{mtheorem}

Theorem \ref{MRt2} is proved in section \ref{s2}.

\begin{remark}
Assuming the hypotheses of Theorem \ref{MRt1}, we have that for $\e=0$ its solutions starting in $C$, i.e. straight lines $\{(t,z):\,t\in\R\}$ for $z\in C$, reaches the set of discontinuity only at its crossing region.  This fact is not necessarily true when we assume the hypotheses of Theorem \ref{MRt2}, because the crossing hypothesis holds only for $\e>0$. This is the main difference between Theorems \ref{MRt1} and \ref{MRt2}. Nevertheless we shall see that to prove both theorems we just have to guarantee that the map $h_z:t\in[0,T]\mapsto h(t,z)$ for $z\in C$ has only isolated zeros. After that the proof follows similarly for both theorems.
\end{remark}

\begin{proposition}\label{hiia}
Assume that $\p_t h(t,x) \neq0$ for each $(t,x)\in\Sigma$. Then hypothesis $(ii)$ holds.
\end{proposition}

\begin{proposition}\label{hiib}
Assume that for each $p\in\Sigma$ such that $\p_t h(p)=0$ there exists a continuous positive function $\xi_p:\R\times D\rightarrow\R$ for which the inequality
\[
\left(\p_th\langle\p_x h,F_1\rangle+\e\dfrac{\langle\p_x h,F_1\rangle^2- \langle\p_x h,F_2\rangle^2}{2}\right)(t,x)\geq \e\xi_p(t,x)
\]
holds for each $(t,x)\in\Sigma_p$. Then hypothesis $(ii')$ holds.
\end{proposition}

It is worthwile to say that the averaging theory appears as a very useful tool in discontinuous
dynamical systems.
For example, in \cite{LT1}, lower bounds for the maximum number of limit cycles for the $m$--piecewise
discontinuous polynomial differential equations was provided using the averaging theory. In \cite{CCL} the averaging theory was used to study
the bifurcation of limit cycles from discontinuous perturbations of two and four dimensional linear center in $\R^n$. Also, in
\cite{LR}, the averaging theory was applied to study
the number of limit cycles of the discontinuous piecewise linear
differential systems in $\R^{2n}$ with two zones separated by a
hyperplane.

\smallskip

In Theorems \ref{MRt1} and \ref{MRt2} we have extended to general discontinuous
differential systems the ideas used in the previous mentioned papers
for particular discontinuous differential systems.

\smallskip

Now an application of Theorem \ref{MRt1} to a class of
discontinuous piecewise linear differential systems is given. Such systems have been studied recently
by Han and Zhang \cite{HZ}, and Huan and Yuang \cite{HY}, among
other papers. In \cite{HZ} some results about the existence of two
limit cycles appeared, so that the authors conjectured that the
maximum number of limit cycles for this class of piecewise linear
differential systems is exactly two. This conjecture is analogous to
Conjecture 1 in the discussion of Tonnelier in \cite{To}. However,
by considering a specific family of discontinuous PWL differential
systems with two linear zones sharing the equilibrium position, in
\cite{HY} strong numerical evidence about the existence of three
limit cycles was obtained, and a proof was provided by Llibre and
Ponce \cite{LP}. This example represents up to now the first
discontinuous piecewise linear differential system with two zones
and $3$ limit cycles surrounding a unique equilibrium. Now we shall
provide a new proof of the existence of these three limit cycles through Theorem \ref{MRt1}.

\smallskip

In polar coordinates $(r,\T)$ given by $x=r \cos \T$ and $y=r
\sin\T$, the planar discontinuous piecewise linear differential
system with two zones separated by a straight line corresponding to
the system studied in the paper \cite{LP} is
\begin{equation}\label{E1}
\frac{dr}{d\T} = F(\T,r)=\left\{ \begin{array}{ll} \e
\dfrac{19}{50}\, r &
\mbox{if $r \cos\T\geq 1$,}\vspace{0.2cm}\\
\e \dfrac{2300 \cos (2 \theta )-4623 \sin (2 \theta )-300}{1500}\, r
& \mbox{if $r \cos\T< 1$,}
\end{array}
\right.
\end{equation}
where we have multiplied the right hand side of the system of
\cite{LP} by the small parameter $\e$. Our main contribution in
this application is to provide the explicit analytic equations
defining the limit cycles of the discontinuous piecewise linear
differential system with two zones \eqref{E1}.

\begin{theorem}\label{t1}
Any limit cycle of the discontinuous piecewise linear differential
system with two zones \eqref{E1} which intersects the straight line
$x=1$ in two points $(r_0,\T_0)$ and $(r_1,\T_1)$ with
$-\pi/2<\T_0<0<\T_1<\pi/2$, $r_k \cos \T_k=1$ for $k=0,1$, and
$r_0>1$ and $\T_1$ must satisfy the following two equations
\begin{equation}\label{E2}
\begin{array}{l}
\exp\left( \dfrac{19 (\T_1-\T_0)}{50} \right) r_0
\cos \T_1-1=0, \vspace{0.3cm}\\
\dfrac{19 (\T_1-\T_0)}{50}+\dfrac{1}{5} \arctan\left(\dfrac{1}{15}
\sec \T_0 (23 \cos\T_0-100 \sin \T_0)\right)\vspace{0.2cm}\\
\quad -\dfrac{1}{5} \arctan\left(\dfrac{1}{15} \sec \T_1 (23 \cos
\T_1-100 \sin \T_1)\right)\vspace{0.2cm}\\
\quad -\dfrac{1}{2} \log (|4623 \cos (2 \T_0)+2300 \sin (2
\T_0)-5377|)\vspace{0.2cm}\\
\quad +\dfrac{1}{2} \log (|4623 \cos (2 \T_1)+2300 \sin (2
\T_1)-5377|)-\dfrac{2 \pi }{5}=0,
\end{array}
\end{equation}
where $\T_0= \arccos(1/r_0)-\pi$, and the determination of the
$\arctan$ is in the interval $(-\pi/2,\pi/2)$ and of the $\arccos$
in the interval $(0,\pi)$.

On the other hand, any limit cycle of system \eqref{E1} which
intersects the straight line $x=1$ in two points $(r_0,\T_0)$ and
$(r_1,\T_1)$ with $0<\T_0<\T_1<\pi/2$, $r_k \cos \T_k=1$ for
$k=0,1$, and $r_0>1$ and $\T_1$ must satisfy the equations
\eqref{E2}, but now in both equations $\T_0= \arccos(1/r_0)$.
\end{theorem}

We recall that a {\it limit cycle} of system \eqref{E1} is an
isolated periodic orbit of that system in the set of all periodic
orbits of the system. It is well known that the study of the limit
cycles of the differential systems in dimension two is one of the
main problems of the qualitative theory of differential systems in
dimension two, see for instance the surveys of Ilyashenko \cite{Il}
and Jibin Li \cite{JL}.

\smallskip

In fact, as we shall see in the proof of Theorem \ref{t1} equations
\eqref{E2} have three solutions, providing the three limit cycles of
Figure \ref{threechinesecycles}.

\begin{figure}[t]
    \begin{center}
    \psfrag{x1}{$x=1$}
   \psfig{file=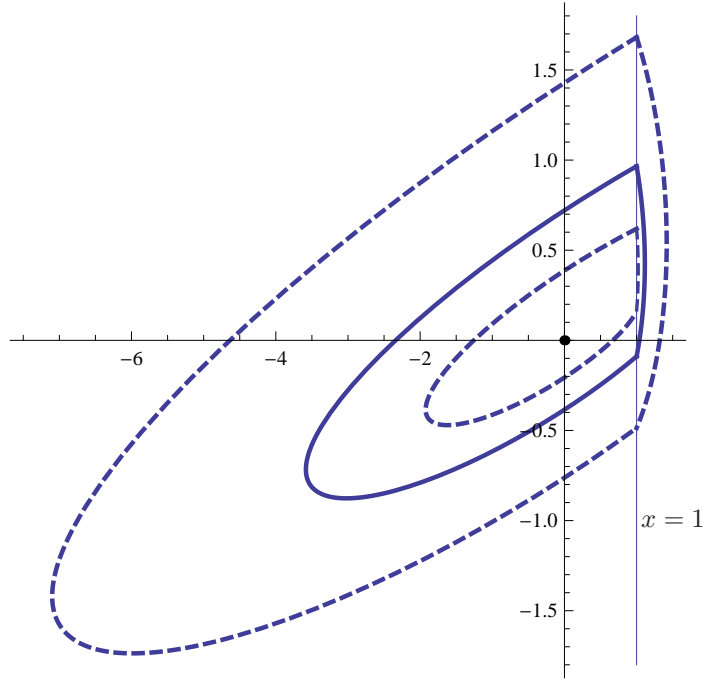,width=9cm}\\
    \end{center}
        \caption{The three limit cycles surrounding the origin.}
        \label{threechinesecycles}
\end{figure}

\section{Proofs of Propositions \ref{hiia} and \ref{hiib} and Theorems \ref{MRt1} and \ref{MRt2}}\label{s2}

\begin{proof}[Proof of Proposition \ref{hiia}]
System \eqref{MRs1} can be written as the autonomous system
\[
\begin{pmatrix}\tau'\\x'\end{pmatrix}=
\begin{cases}
X(\tau,x)\quad \textrm{if}\quad h(\tau,x)>0,\\
Y(\tau,x)\quad \textrm{if}\quad h(\tau,x)<0,
\end{cases}
\]
in $\R\times D$, where
\[
\begin{aligned}
\displaystyle
&X(\tau,x)=\begin{pmatrix}1\\\e(F_1(\tau,x)+F_2(\tau,x))+\e^2
(R_1(\tau,x,\e)+R_2(\tau,x,\e))\end{pmatrix}, \vspace{0.2cm}\\
\displaystyle
&Y(\tau,x)=\begin{pmatrix}1\\\e(F_1(\tau,x)-F_2(\tau,x))+\e^2
(R_1(\tau,x,\e)-R_2(\tau,x,\e))\end{pmatrix}.
\end{aligned}
\]
So
\begin{equation}\label{cro}
\begin{array}{rl}
(Xh)(Yh)=&\langle\nabla h,X\rangle\langle\nabla h,Y\rangle\\
=&(\p_th)^2+\e2\p_th\langle\nabla_x h,F_1\rangle\\
&+\e^2\left(2\p_th\langle\nabla_x h,R_1\rangle+\langle \nabla_xh,F_1\rangle^2-\langle \nabla_xh,F_2\rangle^2\right)\\
&+\e^32\left(\langle\nabla_xh,F_1\rangle\langle\nabla_xh,R_1\rangle-\langle\nabla_xh,F_2\rangle\langle\nabla_xh,R_2\rangle\right)\\
&+\e^4\left(\langle \nabla_xh,R_1\rangle^2-\langle \nabla_xh,R_2\rangle^2\right).
\end{array}
\end{equation}
Let $x(t,z)$ be the solution of the system \eqref{MRs1} such that $x(0,z)=z$. Fixed an open bounded subset $U\subset D_0$ we define the compact subset $K=\{(t,x(t,z)):\,(t,z)\in [0,T]\times\ov U\}\subset [0,T]\times D$.

Hence, we can choose $|\e_0|>0$ sufficiently small such that $(Xh)(Yh)(t,x)>0$ for every $(t,x)\in K\cap\Sigma$ and $\e\in(-\e_0,\e_0)$. Indeed $\left(\p_th(t,x)\right)^2$ is a continuous positive function in $\R\times D$ so there exists $\kappa_0>0$ such that $\left(\p_th(t,x)\right)^2>\kappa_0$ for every $(t,x)\in K$.
\end{proof}

\begin{proof}[Proof of Proposition \ref{hiib}]
Consider the notation of the proof of Proposition \ref{hiia}. From \eqref{cro} we also conclude that
\[
\begin{array}{CL}
(Xh)(Yh)=&\left(\p_th+\e^2\langle\nabla_x h,R_1\rangle\right)^2\\
&+2\e\left(\p_th\langle\nabla_x h,F_1\rangle+\e\dfrac{\langle \nabla_xh,F_1\rangle^2-\langle \nabla_xh,F_2\rangle^2}{2}\right)+\e^3\CO(1).\\
\end{array}
\]
Now, we note that $K\cap\Sigma$ has a finite number of connected components, since $\Sigma$ is a regular manifold in $[0,T]\times D$. So we can choose a finite subset $\{p_1,p_2,\ldots,p_m\}\subset\Sigma$ such that $\p_th(p_i)=0$ for $i=1,2,\ldots,m$, $\Sigma_{p_i}\cap\Sigma_{p_j}=\emptyset$ for $i\neq j$, and $\p_th(t,x)\neq0$ for every $(t,x)\in\Sigma\backslash \left(\Sigma_{p_1}\cup\Sigma_{p_2}\cdots\Sigma_{p_m}\right)$. Thus for $i=1,2,\ldots,m$
\[
(Xh)(Yh)(t,x)\geq\e\xi_{p_i}(t,x)+\e^3\CO(1),
\]
for every $(t,x)\in\Sigma_{p_i}$. We can choose then $\e_i>0$ sufficiently small such that $(Xh)(Yh)(t,x)>0$ for every $(t,x)\in K\cap\Sigma_{p_i}$ and $\e\in(0,\e_i)$. Indeed $\xi_{p_i}$ is a continuous positive function in $\R\times D$ so there exists $\kappa_i>0$ such that $\xi_{p_i}(t,x)>\kappa_i$ for every $(t,x)\in K$. Moreover, by the proof of Proposition \ref{hiia}, we can choose $|\e_0|>0$ sufficiently small such that $(Xh)(Yh)(t,x)>0$ for every $(t,x)\in K\cap\Sigma\backslash \left(\Sigma_{p_1}\cup\Sigma_{p_2}\cup\cdots\cup\Sigma_{p_m}\right)$ and $\e\in(-\e_0,\e_0)$.

Hence, choosing $\bar{\e}=\min\{\e_0,\e_1,\ldots,\e_m\}$ we conclude that $(Xh)(Yh)(t,x)>0$ for every $(t,x)\in K\cap\Sigma$ and $\e\in(0,\bar{\e})$
\end{proof}



For proving Theorems \ref{MRt1} and \ref{MRt2} we need some preliminary lemmas. 
As usual $\mu$ denotes the {\it Lebesgue Measure}.

\smallskip

The hypotheses $(ii)$ and $(ii')$ of Theorem \ref{MRt1} and \ref{MRt2} respectively make assumptions on the Brouwer degree of the averaged function $f$. So we need to show that the function $f$ is continuous in order that the Brouwer degree will be well defined, for more details see Appendix A.

\begin{lemma}\label{MRl2}
The averaged function \eqref{MRf1} is continuous in $C$.
\end{lemma}

\begin{proof}
First of all we note that either the hypothesis $(ii)$ Theorem \ref{MRt1} or $(ii')$ of Theorem \ref{MRt2} implies that the map $h_{z}:t\mapsto h(t,z)$ has only isolated zeros for $z\in C$. Because the constant function $t\mapsto z$ is the solution of system \eqref{MRs1} for $\e=0$, thus, from hypothesis $(ii)$ of Theorem \ref{MRt1}, it reaches the set of discontinuity only at its crossing region. From hypothesis $(ii')$ of Theorem \ref{MRt2} this conclusion is immediately.

For $z\in C$ we define the sets $A^+_z=\{t\in[0,T]:\, h(t,z)>0\}$, $A^-_z=\{t\in[0,T]:\, h(t,z)<0\}$, and $A^0_z=\{t\in[0,T]:\, h(t,z)=0\}$. We note that $\mu\left(A^0(z)\right)=0$, since the map $h_{z}:t\mapsto h(t,z)$ has only isolated zeros for $z\in C$. Moreover $[0,T]=\ov {A^+(z)\cup A^-(z)}$.

Now, fix $z_0\in C$, for $z\in C$ in some
neighborhood of $z_0$, we estimate
\[
\begin{array}{RCL}
|f(z)-f(z_0)|&\leq&\int_0^T|F_1(t,z_0)-F_1(t,z)|dt\\
&&+\int_0^t|\sgn(h(t,z_0))F_2(t,z_0)-\sgn(h(t,z))F_2(t,z)|dt\\
&\leq& TL|z_0-z|\\
&&+\int_{A^+_{z_0}\cap A^+_z}\?|F_2(t,z_0)-F_2(t,z)|dt+\int_{A^-_{z_0}\cap A^-_z}\?|F_2(t,z_0)-F_2(t,z)|dt\\
&&+\int_{A^+_{z_0}\cap A^-_z}\?|F_2(t,z_0)+F_2(t,z)|dt+\int_{A^-_{z_0}\cap A^+_z}\?|F_2(t,z_0)+F_2(t,z)|dt\\
&\leq& 3T|z_0-z|+\left(\mu\left(A^+_{z_0}\cap A^-_z\right)+\mu\left(A^-_{z_0}\cap A^+_z\right)\right)M,
\end{array}
\]
where $L$ is the Lipschitz constant of $F_1$ and $M=\max\{F_2(t,z):\,(t,z)\in[0,T]\times \ov U\}<\infty$. It is easy to see that $\mu\left(A^+_{z_0}\cap A^-_z\right)\to 0$ and $\mu\left(A^-_{z_0}\cap A^+_z\right)\to 0$ when $z\to z_0$. So the lemma is proved.
\end{proof}

Let $a$ be the point in hypothesis $(iii)$ of Theorems \ref{MRt1} and \ref{MRt2}. By
Lemma \ref{MRl2} there exists a neighborhood $U\subset C$ of $a$ such that
$f$ is continuous in $U$. Hence, by Theorems \ref{ApAt1} and
\ref{ApAt2} (see Appendix A), there exists a unique map that
satisfies the properties of the Brouwer degree for the function
$f(z)$ with $z\in \overline{U}$, because $0\notin f(\partial U)$.
This map is denoted by $d_B(f,U,0)$.

\begin{lemma}\label{MRl4}
For $|\e|>0$ (or $\e>0$) sufficiently small the solutions of system \eqref{MRs1} (in the sense of Filippov) starting in $C$ are uniquely defined.
\end{lemma}

To prove Lemma \ref{MRl4} we will need the following proposition,
that has been proved in Corollary 1 of section 10 of chapter 1 of
\cite{F}. Define
\[
\begin{array}{L}
S^+=\{(t,x)\in\R\times D:\,h(t,x)>0\},\\
S^-=\{(t,x)\in\R\times D:\,h(t,x)<0\}.
\end{array}
\]
Note that $\R\times D=S^-\cup \Sigma\cup S^+$.

\begin{proposition}\label{MRp1}
For every point of the manifold $\Sigma$ where $(Xh)(Yh)>0$, there is a
unique solution passing either from $S^-$ into $S^+$, or from  $S^+$
into $S^-$.
\end{proposition}

\begin{proof}[Proof of Lemma \ref{MRl4}]
The proof follows immediately from Proposition \ref{MRp1} and hypothesis $(ii)$.
\end{proof}

Instead of working with the discontinuous differential system
\eqref{MRs1} we shall work with the continuous differential system
\begin{equation}\label{MRs1a}
x'(t)=\e F_{\de}(t,x)+\e^2R_{\de}(t,x,\e),
\end{equation}
where
\[
\begin{aligned}
&F_{\de}(t,x)=F_1(t,x)+\phi_{\de}(h(t,x))F_2(t,x),\\
&R_{\de}(t,x,\e)=R_1(t,x,\e)+\phi_{\de}(h(t,x))R_2(t,x,\e),
\end{aligned}
\]
and $\phi_{\de}:\R\rightarrow\R$ is the continuous function defined
in \eqref{PKf1} and \eqref{e1}, and satisfying \eqref{e2}.

\smallskip

For system \eqref{MRs1a} the averaged function is defined as
\begin{equation*}\label{MRf1a}
f_{\de}(z)=\int_0^T F_{\de}(t,x) dt.
\end{equation*}

We need to guarantee that hypothesis $(i)$ of Theorem \ref{ApBt1} (see
appendix A) holds for the functions $F_{\de}$ and $R_{\de}$. For
this purpose we prove the following lemma.

\begin{lemma}\label{PKr1}
For $\de\in (0,1]$ the function $\phi_{\de}: \R \to \R$ defined in
\eqref{PKf1}  with $\phi$ given by \eqref{e1} is globally
$1/\de$--Lipschitz; i.e. for all $u_1, u_2\in \R$ we have that
$|\phi_{\de}(u_1)- \phi_{\de}(u_2)|\leq (1/\de)|u_1-u_2|$.
\end{lemma}

\begin{proof}
If $u_1\leq -\de< \de\leq u_2$, then
$|\phi_{\de}(u_1)-\phi_{\de}(u_2)|= 2 = (1/\de) 2\de\leq
(1/\de)|u_1-u_2|$.

\smallskip

If $u_1, u_2\leq -\de$ or  $u_1, u_2\geq \de$, then
$|\phi_{\de}(u_1)-\phi_{\de}(u_2)|= 0 \leq (1/\de)|u_1-u_2|$.

\smallskip

Assume that $u_1\in(-\de,\de)$. If $|u_2|<\de$, then
$|\phi_{\de}(u_1)-\phi_{\de}(u_2)|=(1/\de)|u_1-u_2|$; and if
$|u_2|\geq\de$, then $|\phi_{\de}(u_1)-\phi_{\de}(u_2)|\leq
\max\{|1/\de|,|1/u_2|\} |u_1-u_2|\leq(1/\de)|u_1-u_2|$. This
completes the proof of the lemma.
\end{proof}

\begin{proposition}\label{MRl3}
For $\de\in (0,1]$ the functions $F_{\de}$ and $R_{\de}$ are locally
Lipschitz with respect to the variable $x$.
\end{proposition}

\begin{proof}
Let $K\subset D$ be a compact subset. Denote $M=\sup\{|F_2(t,x)|:
(t,x)\in[0,T]\times K\}$, which is well defined by continuity of the
function $(t,x)\mapsto|F_2(t,x)|$ and compactness of the set
$[0,T]\times K$. For $x_1$ and $x_2$ in $K$ where $F_1$ and $h$ are
locally Lipschitz and by Lemma \ref{PKr1}, we have
\[
\begin{array}{RCL}
|F_{\de}(t,x_1)-F_{\de}(t,x_2)|&=&|F_1(t,x_1)-F_1(t,x_2)\\
&&+\phi_{\de}(h(t,x_1))F_2(t,x_1)-\phi_{\de}(h(t,x_2))F_2(t,x_2)|\\
&\leq&|F_1(t,x_1)-F_1(t,x_2)|\\
&&+|\phi_{\de}(h(t,x_1))F_2(t,x_1)-\phi_{\de}(h(t,x_2))F_2(t,x_2)|\\
&\leq& L|x_1-x_1|+|\phi_{\de}(h(t,x_1))||F_2(t,x_1)-F_2(t,x_2)|\\
&&+|F_2(t,x_2)||\phi_{\de}(h(t,x_1))-\phi_{\de}(h(t,x_2))|\\
&\leq& 2L|x_1-x_2|+\dfrac{M}{\de}|h(t,x_1)-h(t,x_2)|\\
&\leq&\left(2L+\dfrac{ML}{\de}\right)|x_1-x_2|=L_{\de}|x_1-x_2|.
\end{array}
\]
Here $L$ is the maximum between the Lipschitz constant of the functions $F_1$ and $F_2$.

The proof for $R_{\de}$ is analogous.
\end{proof}

Now we are ready to prove Theorems \ref{MRt1} and \ref{MRt2}. We shall prove only theorem \ref{MRt1}. The proof of Theorem \ref{MRt2} is completely analogous.

\begin{proof}[Proof of Theorem \ref{MRt1}]
We will study the Poincar\'{e} maps for the discontinuous differential
system \eqref{MRs1} and for the continuous differential system
\eqref{MRs1a}. For each $z\in C$, let $x(t,z,\e)$ denote the
solution (in the sense of Filippov) of system \eqref{MRs1} such that $x(0,z,\e)=z$; and let $x_{\de}(t,z,\e)$ denote the
solution of system \eqref{MRs1a} such that $x_{\de}(0,z,\e)=z$.
Since all solutions starting in $C$ reaches the set of discontinuity at its crossing region for $|\e|>0$ (or $\e>0$) sufficiently small, it follows that $x_{\de}(t,z,\e)\to x(t,z,\e)$ when $\de\to 0$ for every $(t,z)\in[0,T]\times C$ and for $|\e|>0$ (or $\e>0$) sufficiently small.

\smallskip

Since the differential system \eqref{MRs1a} is $T$--periodic in the
variable $t$, we can consider system \eqref{MRs1a} as a differential
system defined on the generalized cylinder $\mathbb{S}^1\times D$
obtained by identifying $\Sigma=\{(\tau,x):\tau=0\}$ with
$\{(\tau,x):\tau=T\}$, see Figure \ref{MRfig1}. On this cylinder
$\Sigma$ is a section for the flow. Moreover, if $z\in C$ is the
coordinate of a point on $\Sigma$, then we consider the Poincar\'{e} map
$P^{\e}_{\de}(z)=x_{\de}(T,z,\e)$ for the points $z$ such that
$x_{\de}(T,z,\e)$ is defined.

\smallskip

Observe that there exists $\e_0>0$ such that, whenever
$\e\in[-\e_0,\e_0]$, the solution $x_{\de}(t,z,\e)$ is uniquely
defined on the interval $[0,T]$. Indeed, if $(t_z^-,t_z^+)$ is the
maximal open interval for which the solution passing through $(0,z)$
is defined. Now we shall apply the local existence and uniqueness
theorem for the solutions of these differential, see for example
Theorem 1.2.2 of \cite{SV}. Note that we can apply that theorem due
to the result of Proposition \ref{MRl3}. Hence, by the local
existence and uniqueness theorem we have that $t_z^+>h_z$ and
$h_z=\inf\{T,d\backslash m(\e)\}$ where $m(\e)\geq|\e
F_{\de}(t,x)+\e^2R_{\de}(t,x,\e)|$ for all $t\in[0,T]$, for each $x$
with $|x-z|\leq d$ and for every $z\in C$. When $|\e|>0$ (or $\e>0$) is
sufficiently small, $m(\e)$ can be arbitrarily large, in such a way
that $h_z=T$ for all $z\in C$. Hence, for $\e\in[-\e_0,\e_0]$, the
Poincar\'{e} map of system \eqref{MRs1a} is well defined and continuous
for every $z\in C$.

\smallskip

\begin{figure}
\psfrag{A}{$t$}
\psfrag{B}{$\Sigma$}
\psfrag{C}{$\R^n$}
\psfrag{D}{$_{nT\equiv 0}$}
\includegraphics[width=4cm]{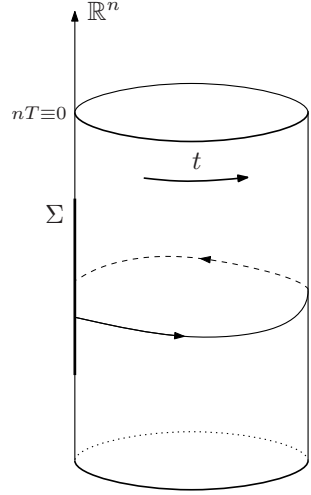}
\vskip 0cm \centerline{} \caption{\small \label{MRfig1} Generalized
cylinder.}
\end{figure}

From the definition of the Poincar\'{e} map $P^{\e}_{\de}(z)$ its
fixed points correspond to periodic orbits of period $T$ of the
differential system \eqref{MRs1a} defined on the cylinder.

\smallskip

We can define in a similar way the Poincar\'{e} map $P^{\e}(z)=x(T,z,\e)$ of the
discontinuous differential system \eqref{MRs1}. The referred Poincar\'{e} map is the composition of the Poincar\'{e} maps of
the continuous differential systems, so for $|\e|>0$ (or $\e>0$) sufficiently small it is well defined and continuous
for every $z\in C$. Again the fixed points of $P^{\e}(z)$ correspond to periodic orbits of the
discontinuous differential system \eqref{MRs1}.

\smallskip

Clearly (from above considerations), for $z\in C$ and for $|\e|>0$ (or $\e>0$) sufficiently small, the pointwise limit of the Poincar\'{e} map
$P^{\e}_{\de}(z)$ of system \eqref{MRs1a}, when $\de\to 0$ is the Poincar\'{e} map $P^{\e}(z)$  of system \eqref{MRs1}.

\smallskip

By definition the continuous differential system \eqref{MRs1a} is
$\CC^2$ in the variable $\e$. So we do the Taylor expansion of the
Poincar\'{e} map of system \eqref{MRs1a} around $\e$ up to order two,
and we get
\begin{equation}\label{e3}
P^{\e}_{\de}(z)=z+ \e f_{\de}(z)+\CO(\e^2),
\end{equation}
where $f_{\de}(z)$ is the averaged function of the continuous
differential system \eqref{MRs1a}, for more details see for instance
\cite{BL}. Due to \eqref{e2} we obtain that the pointwise limit in
$\Sigma$ of the function $f_{\de}$, when $\de\to 0$ is the function
$f$.

\smallskip

Let $a\in U$ be the point satisfying hypotheses $(ii)$ of Theorem
\ref{MRt1}. Therefore $f(z)\neq 0$ for all $z\in\overline{V}
\backslash\{a\}$. Define $f_0=f|_{\ov V}$, we know that $f_0$ is
continuous by Lemma \ref{MRl2}. Then, we consider the continuous
homotopy $\{f_{\de}|_{\ov V}, 0\leq \de \leq 1\}$. We claim that
there exists a $\de_0\in (0,1]$ such that $0\notin f_{\de}(\partial
V)$ for all $\de\in[0,\de_0]$. Now we shall prove the claim.

\smallskip

As usual $\N$ denotes the set of positive integers. Suppose that
there exists a sequence $(z_m)_{m\in \N}$ in $\partial U$ such that
$f_{\frac{1}{m}}(z_m)=0$. As the sequence $(z_m)$ is contained in
the compact set $\partial U$, so there exists a subsequence
$(z_{m_\ell})_{\ell\in\N}$ such $z_{m_\ell}\to z_0\in\partial U$.
Consequently we obtain that $f(z_0)=0$, in contradiction with the
hypotheses $(ii)$ of Theorem \ref{MRt1}. Hence, the claim is proved.

\smallskip

{From} the above claim and the property $(iii)$ of Theorem \ref{ApAt1}
(see Appendix A) we conclude that $d_B(f_{\de},V,0)\neq 0$ for
$0\leq \de\leq\de_0$. Therefore, by the property $(i)$ of Theorem
\ref{ApAt1} we obtain that $0\in f_\de(V)$, so there exists
$a_{\de}\in U$ such that $f_{\de}(a_{\de})=0$. Since, by continuity,
there exists the $\lim_{\de\to 0} a_\de$ and it is a zero of the
function $f_0=f|_U$. This limit is the point $a$ of the hypotheses
$(ii)$ of Theorem \ref{MRt1}, because $a$ is the unique zero of $f_0$
in $U$.

\smallskip

In summary, in order that for every $\de\in (0,\de_0]$ the averaged
function $f_\de$ satisfy the assumptions $(ii)$ of Theorem \ref{ApBt1}
(see Appendix B). So it only remains to show that in $U$ we have
that $f_\de(z)\neq 0$ for all $z\in \ov V\setminus \{a_\de\}$. But
this can be achieved in a complete similar way as we proved the
above claim. Hence, by Proposition \ref{MRl3} for every $\de\in
(0,\de_0]$ the continuous differential system \eqref{MRs1a}
satisfies all the assumptions of Theorem \ref{ApBt1}. Hence, for
$|\e|$ sufficiently small there exists a periodic solution
$x_{\de}(t,\e)$ of the continuous differential system \eqref{MRs1a}
such that $z_{(\de,\e)}:= x_{\de}(0,\e)\to a_{\de}$ when $\e\to 0$.

\smallskip

Now, from \eqref{e3} the point $z_{(\de,\e)}$ is a fixed point of
the Poincar\'{e} map $P^{\e}_{\de}(z)$, i.e. $P^{\e}_{\de}(z_{(\de,
\e)})= z_{(\de,\e)}$.  Since $\lim_{\de\to 0} P^{\e}_{\de}(z)=
P^{\e}(z)$, it follows that $z_\e= \lim_{\de\to 0} z_{(\de,\e)}$ is
a fixed point of the Poincar\'{e} map $P^{\e}(z)$. So, the discontinuous
differential system \eqref{MRs1} has a periodic solution $x(t,\e)$
such that $z_\e=x(0,\e)\to a$ as $\e\to 0$. Therefore the theorem is
proved.
\end{proof}

\section{Application}

In this section we shall prove Theorem \ref{t1}, by applying Theorem
\ref{MRt1} to the discontinuous differential system \eqref{E1}. So,
we must compute the integral \eqref{MRf1}, which for system
\eqref{E1} becomes
\begin{equation}\label{E3}
f(r)= \int_0^{2\pi} F(\T,r)\, d\T,
\end{equation}
where the function $F(\T,r)$ is given in \eqref{E1}.

\smallskip

The solution of the differential system \eqref{E1} in the
half--plane $x= r \cos\T \geq 1$ starting at the point $(r_0, \T_0)$
with $r_0 \cos\T_0=1$ and $\T_0\in (-\pi/2,0)$ is
\[
r(\T)= \exp\left( \dfrac{19 (\T -\T_0)}{50}\right) r_0.
\]
Therefore, at the point $(r_1, \T_1)$ with $r_1 \cos\T_1=1$ and
$\T_1\in (0,\pi/2)$ we have that
\[
\exp\left( \dfrac{19 (\T_1 -\T_0)}{50}\right) r_0 \cos \T_1 = 1.
\]
This equation coincides with the first equation of \eqref{E2}.

\smallskip

Now computing the integral \eqref{E3} we obtain exactly the right
hand side of the second equation of \eqref{E2} multiplied by $r$.
According to Theorem \ref{MRt1} we must find the zeros of this last
expression. Since $r$ cannot be zero the equation for the zeros is
reduced exactly to the second equation of \eqref{E2}. In short, by
Theorem \ref{MRt1} we have proved that a periodic orbit of system
\eqref{E1} intersects the straight line $x=1$ in two points
$(r_0,\T_0)$ and $(r_1,\T_1)$ with $\T_0\in (-\pi/2,0)$, $\T_1\in
(0,\pi/2)$, $r_k \cos \T_k=1$ for $k=0,1$, and $r_0>1$ and $\T_1$
must satisfy the equations \eqref{E2}.

\smallskip

In \cite{LP} it is proved that the discontinuous differential
equation \eqref{E1} has three limit cycles $(i)$ and that the their
points $(r_0,\T_0)$ and $(r_1,\T_1)$ are approximately for the inner
limit cycle of Figure \ref{threechinesecycles}
\begin{equation}\label{E4}
r_0= 1.013330663139..,\quad \T_0=0.162383740477..,  \quad \T_1=
0.5541676264624..;
\end{equation}
for the middle limit cycle of Figure \ref{threechinesecycles}
\begin{equation}\label{E5}
r_0= 1.003945075086..,\quad \T_0=-0.088680876377.., \quad \T_1=
0.768002346543..;
\end{equation}
for the external limit cycle of Figure \ref{threechinesecycles}
\begin{equation}\label{E6}
r_0= 1.111870463116..,\quad \T_0=-0.452434880837.., \quad \T_1=
1.034197922817.. .
\end{equation}
It is easy to check that \eqref{E4}, \eqref{E5} and \eqref{E6}
satisfies the two equations \eqref{E2}. Hence, Theorem \ref{t1} is
proved.

\section*{Appendix A:  Basic results on the Brouwer degree}\label{ApA}

In this appendix we present the existence and uniqueness result from
the degree theory in finite dimensional spaces. We follow the
Browder's paper \cite{B}, where are formalized the properties of the
classical Brouwer degree.

\begin{theorem}\label{ApAt1}
Let $X=\R^n=Y$ for a given positive integer $n$. For bounded open
subsets $V$ of $X$, consider continuous mappings
$f:\overline{V}\rightarrow Y$, and points $y_0$ in $Y$ such that
$y_0$ does not lie in $f(\partial V)$ (as usual $\partial V$ denotes
the boundary of $V$). Then to each such triple $(f,V,y_0)$, there
corresponds an integer $d(f,V,y_0)$ having the following three
properties.
\begin{itemize}
\item[$(i)$] If $d(f,V,y_0)\neq 0$, then $y_0\in f(V)$. If $f_0$
is the identity map of $X$ onto $Y$, then for every bounded open set
$V$ and $y_0\in V$, we have
\[
d\left(f_0\big|_V,V,y_0\right)=\pm 1.
\]

\item[$(ii)$] $($Additivity$)$ If $f:\overline{V}\rightarrow Y$ is
a continuous map with $V$ a bounded open set in $X$, and $V_1$ and
$V_2$ are a pair of disjoint open subsets of $V$ such that
\[
y_0\notin f(\overline{V}\backslash(V_1\cup V_2)),
\]
then,
\[
d\left(f_0,V,y_0\right)=d\left(f_0,V_1,y_0\right)+
d\left(f_0,V_1,y_0\right).
\]

\item[$(iii)$] $($Invariance under homotopy$)$ Let $V$ be a
bounded open set in $X$, and consider a continuous homotopy
$\{f_t:0\leq t\leq 1\}$ of maps of $\overline{V}$ in to $Y$. Let
$\{y_t:0\leq t\leq 1\}$ be a continuous curve in $Y$ such that
$y_t\notin f_t(\partial V)$ for any $t\in[0,1]$. Then $d(f_t,V,y_t)$
is constant in $t$ on $[0,1]$.
\end{itemize}
\end{theorem}

\begin{theorem}\label{ApAt2}
The degree function $d(f,V,y_0)$ is uniquely determined by the three
conditions of Theorem \ref{ApAt1}.
\end{theorem}

\smallskip

For the proofs of Theorems \ref{ApAt1} and \ref{ApAt2} see \cite{B}.

\section*{Appendix B: Basic results on averaging theory}\label{ApB}

In this appendix we present the basic result from the averaging
theory that we shall need for proving the main results of this
paper. For a general introduction to averaging theory see for
instance the book of Sanders and Verhulst \cite{SV}.

\begin{theorem}\label{ApBt1}
We consider the following differential system
\begin{equation}\label{ApBs1}
x'(t)=\e\,F(t,x)+\e^2\,R(t,x,\e),
\end{equation}
where $F:\R\times D\rightarrow\R^n$ and $R:\R\times
U\times(-\e_f,\e_f)\rightarrow\R^n$ are continuous functions,
$T$-periodic in the first variable and $D$ is an open subset of
$\R^n$. We define the averaged function $f:D\rightarrow\R^n$ as
\begin{equation}\label{ApBf1}
f(x)=\int_0^T F(s,x)ds,
\end{equation}
and assume that
\begin{itemize}
\item[$(i)$] $F$ and $R$ are locally Lipschitz with respect to $x$;

\item[$(ii)$] for $a\in D$ with $f(a)=0$, there exist a neighborhood
$V$ of $a$ such that $f(z)\neq0$ for all
$z\in\overline{V}\backslash\{a\}$ and $d_B(f,V,0)\neq0$.
\end{itemize}
Then, for $|\e|>0$ sufficiently small, there exist a $T$--periodic
solution $x(t,\e)$ of the system \eqref{ApBs1} such that $x(0,\e)\to
a$ as $\e\to 0$.
\end{theorem}

Theorem \ref{ApBt1} for studying the periodic orbits of continuous
differential systems has weaker hypotheses than the classical result
for studying the periodic orbits of smooth differential systems, see
for instance Theorem 11.5 of Verhulst \cite{V}, where instead of $(i)$
is assumed that
\begin{itemize}
\item[$(j)$] $F,\,R,\,D_xF,\,D_x^2F$ and $D_xR$ are defined, continuous
and bounded by a constant $M$ (independent of $\e$) in
$[0,\infty)\times D$, $-\e_f<\e<\e_f$;
\end{itemize}
and instead of $(ii)$ it is required that
\begin{itemize}
\item[$(jj)$] for $a\in D$ with $f(a)=0$ we have that $J_{f}(a)\neq 0$,
where $J_{f}(a)$ is the Jacobian matrix of the function $f$ at the
point $a$.
\end{itemize}

\smallskip

For a proof of Theorem \ref{ApBt1} see \cite{BL} section 3.

\section*{Acknowledgements}

The first author is partially supported by a MICINN/FEDER grant
MTM2008--03437, by a AGAUR grant number 2009SGR--0410, by ICREA
Academia and FP7 PEOPLE-2012-IRSES-316338 and 318999. The second
author is partially supported by a FAPESP--BRAZIL grant 2012/10231--7. The
third author is partially supported by a FAPESP--BRAZIL grant
2007/06896--5. The first and third authors are also supported by the
joint project CAPES-MECD grant PHB-2009-0025-PC.

\end{document}